\documentclass[11pt,reqno]{amsart}
\usepackage[skins]{tcolorbox}
\usepackage{mathrsfs}
\usepackage{xcolor}

\usepackage{enumitem}
\usepackage {amssymb,color}
\usepackage{verbatim}
\input xy
\xyoption{all}
\input epsf

\newcommand{\GL}{\mathrm{GL}}

\newcommand{\Sp}{\mathrm{Sp}}
\newcommand{\mic}{{mic}}

\newcommand{\U}{\mathrm{U}}

\newcommand{\field}{\mathbb}
\newcommand{\liealgebra}{\mathfrak}

\newcommand{\C}{{\field C}}
\newcommand{\R}{{\field R}}
\newcommand{\Z}{{\field Z}}
\newcommand{\Q}{{\field Q}}

\renewcommand{\t}{\liealgebra t}

\newcommand{\lam}{\lambda}

\newcommand{\eps}{\epsilon}

\newcommand{\SO}{\mathrm{SO}}

\newcommand{\lra}{\longrightarrow}
\newcommand{\ra}{\rightarrow}

\newcommand{\per}{\textrm{per}}

\newcommand{\con}{\textrm{con}}

\newtheorem{prop}{Proposition}[section]
\newtheorem{cor}[prop]{Corollary}
\newtheorem{lemma}[prop]{Lemma}
\newtheorem{theorem}[prop]{Theorem}

\numberwithin{equation}{section}

\newtheorem{corollary}[prop]{Corollary}

\theoremstyle{definition}

\newtheorem{remark}[prop]{Remark}
\newtheorem{example}[prop]{Example}

\newtheorem{definition}[prop]{Definition}

\newcommand{\IB}{A} 

\newcommand{\frb}{\mathfrak{b}}

\newcommand{\frg}{\mathfrak{g}}
\newcommand{\frh}{\mathfrak{h}}

\newcommand{\frk}{\mathfrak{k}}
\newcommand{\frl}{\mathfrak{l}}

\newcommand{\frn}{\mathfrak{n}}
\newcommand{\fro}{\mathfrak{o}}
\newcommand{\frp}{\mathfrak{p}}
\newcommand{\frq}{\mathfrak{q}}

\newcommand{\frs}{\mathfrak{s}}
\newcommand{\frt}{\mathfrak{t}}
\newcommand{\fru}{\mathfrak{u}}

\newcommand{\bbC}{\mathbb{C}}

\newcommand{\bbQ}{\mathbb{Q}}
\newcommand{\bbR}{\mathbb{R}}

\newcommand{\bbZ}{\mathbb{Z}}

\newcommand{\caC}{\mathcal{C}}

\newcommand{\caL}{\mathcal{L}}

\newcommand{\caO}{\mathcal{O}}
\newcommand{\caP}{\mathcal{P}}

\newcommand{\caV}{\mathcal{V}}

\newcommand{\caX}{\mathcal{X}}

\newcommand{\beq}{\begin{equation}}
\newcommand{\eeq}{\end{equation}}

\def \Ad  {\mathop{\hbox {Ad}}\nolimits}
\def\ad {\mathop{\hbox {ad}}\nolimits}

\renewcommand{\c}{\mathrm{con}}

\begin{document}
\title[]{relating real and p-adic Kazhdan-Lusztig polynomials}
\author{Leticia Barchini}
\address{Department of Mathematics, Oklahoma State University, Stillwater, OK 74078}
\email{leticia@math.okstate.edu}

\author{Peter E.~Trapa}
\address{Department of Mathematics, University of Utah, Salt Lake City, UT 84112-0090}
\email{peter.trapa@utah.edu}

\maketitle

\begin{abstract}
Fix an integral semisimple element $\lambda$ in the Lie algebra $\frg$ of a complex reductive algebraic group $G$.  Let $L$ denote the centralizer of $\lambda$ in $G$ and let $\frg(-1)$ denote the $-1$-eigenspace of $\ad(\lambda)$ in $\frg$.  Under a natural hypothesis (which is always satisfied for classical groups), we embed
 the closure of each $L$ orbit on $\frg(-1)$ into the closure of an orbit of a symmetric subgroup $K$ containing $L$ on a partial flag variety for $G$.
We use this to relate the local intersection homology of the latter orbit closures to the former orbit closures.  This, in turn,
relates multiplicity matrices for split real and $p$-adic groups. We also describe relationships between
``microlocal packets'' of representations of these groups.
\end{abstract}

\section{introduction}
The main result of this paper, Theorem \ref{t:main}, relates certain Kazhdan-Lusztig polynomials that arise in the representation theory of real and $p$-adic groups.  Since these polynomials encode the multiplicities of irreducible representations in standard representation, we thus relate these two kinds of multiplicities.  Under favorable circumstances (which
are always satisfied for $\GL(n)$ and $\Sp(2n)$,  for example), our results imply that the decomposition matrix for certain unipotent representations of a split $p$-adic group is a submatrix (in a variety of different ways) of the decomposition matrix for representations of a split real group.   

In more detail, suppose $\lambda$ is a semisimple element in the Lie algebra $\frg$ of a complex reductive algebraic group $G$.  Let $L$ denote the centralizer in $G$ of $\lambda$.  Then, for $c \in \bbC^\times$, $L$ acts with finitely many orbits on the $c$-eigenspace of $\ad(\lambda)$ \cite{Vi}. Thus one can consider local intersection homology Poincar\'e polynomials for the closures of $L$ orbits.  (An elementary argument reduces matters to the case of $\lambda$ integral and $c=-1$ and we will consider this case henceforth; see Remark \ref{r:integral}.)  Lusztig \cite{lu:blocks,lu:alg} established a finite
effective algorithm to compute these polynomials. Roughly speaking, their values at 1 give multiplicities of irreducible unipotent representations in standard representations of the split $p$-adic form of the Langlands dual of $G$ \cite{lu:cls2}.

On the other hand, let $K$ denote the identity component of the fixed points in $G$ of the automorphism $\theta$ obtained by conjugation by $\exp(i\pi\lambda)$.  Then $K$ acts with finitely
many orbits on the partial flag variety $\caP$ consisting of conjugates of the sum of the nonnegative eigenspaces of $\ad(\lambda)$.  Once again, one can consider local intersection homology Poincar\'e polynomials.  Vogan  \cite{v:ic3}
described a finite effective algorithm to compute them which has been implemented in the software package {\tt atlas}.
Evaluating these polynomials at 1 gives multiplicities of irreducibles in standard representation
of the identity component of the split real form of the Langlands dual of $G$ \cite{vogan:ic4,abv}.

The algorithms of \cite{lu:alg} and \cite{v:ic3} are completely different. Nonetheless, Theorem \ref{t:main} says that certain polynomials that they compute separately are the same.  Under a certain hypothesis (see \eqref{e:key}) we define a map $\eps$, a kind of truncated exponential map, from the $-1$-eigenspace $\frg(-1)$ of $\ad(\lambda)$ to $\caP$, and use it to define an injection of $L$ orbits on $\frg(-1)$ to $K$ orbits on $\caP$; see Definition \ref{d:eps}.
This gives rise to a restriction $\varphi$ of irreducible local $K$-equivariant local system on $\caP$ to irreducible $L$ equivariant local systems on $\frg(-1)$.  We prove in Theorem \ref{t:main} that
\begin{equation}
\label{e:intro}
P_{\varphi(\psi), \varphi(\gamma)} = P_{\psi,\gamma};
\end{equation}
here the polynomial on the left-hand side is relevant for $p$-adic group representations, and the polynomial on the right-hand side is relevant for real group representations.  
In terms of representation theory, $\varphi$ can be thought of as matching an irreducible unipotent representation of a split $p$-adic group with an 
irreducible Harish Chandra module for a split real group,
and \eqref{e:intro}  implies that their respective multiplicities in certain standard modules (also matched by $\varphi$) are the same.  

Equation \eqref{e:intro} generalizes the main geometric result for $\GL(n)$ of Ciubotaru-Trapa \cite{ct}; see Remark \ref{r:ct}.   In Remark \ref{r:O}, we also give analogous results equating the $p$-adic polynomials with Kazhdan-Lusztig polynomials arising from category $\caO$, generalizing results of Zelevinsky \cite{zel} for $\frg\frl(n)$ to all classical groups.

There are several important hypotheses to highlight.  As we indicated above, the existence of the map $\eps$ depends
on \eqref{e:key}.  As explained in Example \ref{ex:classical}, \eqref{e:key} always holds in the classical groups, but Example \ref{ex:f4}
shows that it cannot hold in certain exceptional cases.  Moreoever, when \eqref{e:key} holds, the definition of $\eps$ depends on a certain choice of ordering of the set $\mathscr{P}$ in Definition \ref{d:eps}.  Different choices lead to different maps $\eps$, and therefore to different maps $\varphi$ on local systems.  Nonetheless \eqref{e:intro} hold for all such choices.  In other words, depending on the choices made, we possibly identify the $p$-adic polynomial on the left-hand side of \eqref{e:intro} with different instances of the real polynomial on the right-hand side. 
The dependence on the choice of ordering is perhaps disappointing, since one might have hoped for a canonical relationship. On the other hand, in practice different choices  lead to different matchings that reveal interesting nontrivial coincidences among the polynomials in question.  

Note also that we have assumed $K$ to be the identity component of the fixed points of $\theta$.  This is perhaps unnatural from the point of view of representations of algebraic groups, and requires explanation.  If we had instead worked with the potentially disconnected $K' = G^\theta$, the orbits of $K'$ on $\caP$ can be reducible, and different irreducible components can contribute to the local intersection cohomology in ways that make the left-hand side of \eqref{e:intro} a {\em summand} of the right-hand side.   In any particular case, this is tractable to understand, but general statements in the presence of this kind of disconnectedness are somewhat cumbersome.  See Remark \ref{r:disconnected}.  

Finally, in addition to matching local intersection homology polynomials, Theorem \ref{t:main} shows that $\varphi$ also matches microlocal geometric information.  A number of interesting consequences for ABV micro-packets of representations are sketched in Section \ref{s:abv}. 

%

\section{matching of orbits}
\label{s:prelim}

Let $G$ be a complex connected reductive algebraic group with Lie algebra $\frg$.  
Fix a Borel subalgebra $\frb = \frt \oplus \frn$ and fix a semisimple element $\lambda \in \frt$ which is integral and weakly dominant in
the sense that the pairing of $\lambda$ with any coroot for $\t$ in $\frn$ is a non-negative integer.  (See Remark \ref{r:integral} for a discussion of the nonintegral case.) For each $i \in \bbZ$, write
\[
\frg(i) = \left \{ x \in \frg \; | \; [\lambda, x] =ix \right \},
\]
the $i$-eigenspace for $\ad(\lambda)$.  Set $\frl{}$ equal to $\frg(0)$, and let $L$ denote the centralizer in $G$ of $\frl$.  Then $L$ acts with finitely many
orbits on each $\frg(i)$ \cite{Vi}.

 Let
\[
\fru{} = \bigoplus_{i > 0} \frg(i) \qquad \text{and} \qquad \overline{\fru}{} = \bigoplus_{i < 0} \frg(i);
\]
and
\[
\frp{} = \frl{} \oplus \fru{}.
\]
Thus $\frp{}$ is a parabolic subalgebra containing $\frb$.  Write $\caP{}$ for the variety of conjugates of $\frp{}$.  Then $\caP \simeq G/P{}$ where
$P=LU$ is the centralizer in $G$ of $\frp{}$.

Let $y(\lam)$ denote $\exp(i\pi \lam)$.  Since $\lambda$ is integral, the square of $y(\lam)$ is central.  Conjugation by $y(\lam)$ therefore defines an involution $\theta$ of $G$.  Let $K$ denote the identify component of its fixed points.   Note that the Lie algebra of $K$ is
\[
\frk =\bigoplus_i \frg(2i)
\]
and $K$ contains $L$ by definition.

\begin{example}
\label{ex:K}
Let $G = \GL(n,\C)$ and let $\frt$ denote the diagonal Cartan subalgebra.  After a central shift, the integrality of $\lam$ implies that we may assume $\lam$ consists of integer entries.  Then $K\simeq \GL(p,\bbC) \times \GL(q,\bbC)$ where $p$ is the number of even entries of $\lambda$ and $q$ is the number of odd entries.  The symmetric pair $(G,K)$ corresponds to the real group $\U(p,q)$.  This is the setting of \cite{ct}.

Next let $G = \Sp(2n,\bbC)$ and let $\frt$ denote the diagonal Cartan subalgebra in the standard realization.  For $\lam$ to be integral, either all of its entires are integers, or else they are all integers shifted by $1/2$.  If $\lam$ consists of all half-integers, then $K \simeq \GL(n,\bbC)$, corresponding to the real group $\Sp(2n,\bbR)$.  If $\lam$ consists of all integers, then $K \simeq \Sp(p,\bbC) \times \Sp(q,\bbC)$ where $p$ is the number of even entries of $\lambda$ and $q$ is the number of odd entries.  This case corresponds to the real group $\Sp(p,q)$.\qed
\end{example}

\bigskip

Recall that our goal is to relate $L$ orbits on $\frg(-1)$ and $K$ orbits on $\caP$.  As mentioned in the introduction,
we do this using a kind of truncated exponential map.  In order to define the map, we need to introduce certain hypotheses which we now describe.

Let $\mathscr{P}$ denote a collection of parabolic subalgebras each of which properly contains $\frp{}$.  Choose an order on the elements of $\mathscr{P}$ and write 
\[
\mathscr{P} = \{ \frp_1, \dots, \frp_\ell \}.
\]
For each $i$, write the Levi decomposition as $\frp_i = \frl_i \oplus \fru_i$.  Since $\frp_i$ contains $\frp{}$, 
\[
\frl_i \cap \overline{\fru}{}
\]
is the nilradical of a parabolic subgroup of $\frl_i$.  We will be interested in imposing the following hypotheses on the collection $\mathscr{P}$: first,
that 
\begin{equation}
\label{e:key0}
\frl_i \cap \frg(-1) \neq \{0\}; 
\end{equation}
and second that
\begin{equation}
\label{e:key}
\frg(-1) = \bigoplus_{i = 1}^\ell \left [ \frl_i \cap \overline{\fru}{} \right ].
\end{equation}

\begin{definition}
\label{d:eps}
Fix an ordered collection $\mathscr P = \{\frp_1, \dots, \frp_\ell\}$ satisfying \eqref{e:key0} and \eqref{e:key}.  (Such a collection $\mathscr{P}$ always exists
if $G$ is classical, but need not always exist in the exceptional cases; see Examples \ref{ex:rho-check}--\ref{ex:g2} below.)
Define
\[
\epsilon \; : \; \frg(-1) \lra \caP
\]
by writing $x\in \frg(-1)$ as $x_1+\cdots + x_\ell$ according to \eqref{e:key} and
setting
\[
\epsilon(x) = \exp(x_1)\exp(x_2) \cdots \exp(x_\ell)\cdot \frp.
\]
Fix an orbit $\caO$ of $L$ on $\frg(-1)$.  Since
$\epsilon$ is $L$ equivariant and since $L \subset K$, $K \cdot \epsilon (\caO)$ consists
of a single $K$ orbit which we call $Q_\caO$.    The assignment
\begin{equation}
\label{e:QO}
\caO \mapsto Q_\caO
\end{equation}
defines an injection of $L$ orbits on $\frg(-1)$ into $K$ orbits on $\caP$.  Note that the definition of $\eps$ (and hence $Q_\caO$) depends on the choice of ordering of the elements of $\mathscr P$.  Let
\begin{equation}
\label{e:Y}
Y = \bigcup_{\caO} Q_\caO
\end{equation}
where $\caO$ ranges over all orbit of $L$ on $\frg(-1)$.  Thus, by definition, $Y$ is the $K$ saturation
of the image of 
$\epsilon$.\qed
\end{definition}

\bigskip

\noindent
The next examples investigate some instances when \eqref{e:key} holds (or cannot hold).

\begin{example}
\label{ex:rho-check}
Suppose $\lam = \rho^\vee$, the half-sum of the coroots corresponding to the roots of $\frt$ in $\frn$.  As $\rho^\vee$ is regular, $\frp{} = \frb$.   Let $\mathscr{P}$ denote the set of parabolic subalgebras are minimal among those that properly contain $\frb$. If we enumerate the simple roots of $\t$ in $\frn$ as $\alpha_1, \dots, \alpha_\ell$,
then we can enumerate $\mathscr{P}$ as $\frp_1, \dots \frp_\ell$ with
\[
\frp_i = \frg_{-\alpha_i} \oplus \frb,
\]
where $\frg_{-\alpha_i}$ is the root space for $-\alpha_i$ in $\frg$.  Thus 
\[
 \frl_i \cap \overline{\fru}{} = \frg_{-\alpha_i}
 \]
Since $\frg(-1)$
is the span of the negative simple root spaces, $\mathscr{P}$ satisfies \eqref{e:key0} \and \eqref{e:key}.  This is the setting of \cite{bt}.\qed
\end{example}

\begin{example}
\label{ex:classical}
Suppose $G$ is a classical group, and $\lambda$ is an arbitrary (possibly singular) integral element.    Let $\mathscr{P}$ denote the set of parabolic subalgebras are minimal with respect to the properties of: (1) properly containing $\frp$; and (2) having a levi factor that meets $\frg(-1)$ in a nonzero subspace.  
Then $\mathscr{P}$ always satisfies \eqref{e:key0} and \eqref{e:key}.  

To see this, first take the case of $G = \GL(n,\bbC)$, let $\frb$ 
be upper-triangular matrices with $\frt$ the diagonal ones.  Fix
\begin{equation}
\label{e:lam}
\lambda = (\overbrace{a_1, \cdots, a_1}^{n_1}, \overbrace{a_2, \cdots, a_2}^{n_2}, \dots, \overbrace{a_k, \cdots, a_k}^{n_k})
\end{equation}
with each $a_i$ an integer, $a_i > a_{i+1}$ and $n=n_1+\dots+n_k$.  Then $\frp{}$ consists of block-upper triangular matrices with diagonal blocks of size $n_1, \dots, n_k$.  Set 
\[
S = \{ j \; | \; a_j-a_{j+1} = 1\}.
\]
The space $\frg(-1)$ identifies with a subspace of 
block matrices just below the diagonal,
\begin{equation}
\label{e:block}
\frg(-1) \simeq \bigoplus_{j\in S}\mathrm{Mat}_{n_{j+1},n_{j}} 
\end{equation}
A minimal parabolic $\frp_j$containing $\frp$ is obtained by enlarging the adjacent Levi factors $\frg\frl(n_{j}) \oplus \frg\frl(n_{j+1})$ to $\frg\frl(n_{j}+n_{j+1})$, and $\frp_j$ has nonzero intersection with $\frg(-1)$ when $j\in S$.  
Thus $\mathscr P = \{\frp_j \; | \; j \in S\}$.
The element
$\frp_j$ has
\[
\frl_j\cap \overline{\fru}{}
\]
consisting of the block lower-triangular matrices in $\frg\frl(n_{j}+n_{j+1})$, namely $\mathrm{Mat}_{n_{j+1},n_{j}}$.  Comparing with \eqref{e:block}, one sees \eqref{e:key} holds. 
Note that $\lam = \rho^\vee$ in the previous example is simply the case when all block sizes are 1.

The other classical cases are similar.  For example, suppose $G = \Sp(2n,\bbC)$ and 
$\lam$ is as in \eqref{e:lam} in standard coordinates.  Suppose all entries $a_i$ are half-integers (but not integers). Then $\frp$ has Levi factor $\frg\frl(n_1)\oplus \cdots \oplus \frg\frl(n_k)$.  Once again we obtain a minimal parabolic $\frp_j$ containing $\frp$ by expanding adjacent
factors $\frg\frl(n_j) \oplus \frg\frl(n_{j+1})$ in the Levi factor for $\frp$ to $\frg\frl(n_j+n_{j+1})$, and $\frp_j$ meets $\frg(-1)$ nontrivially if $j \in S$ as defined above.  Then $\mathscr P$ consists of the $\frp_j$ for $j \in S$ plus possibly one other element: if $a_k=1/2$, then $\mathscr P$ contains $\frp_\circ$ whose Levi factor is $\frg\frl(n_1)\oplus \cdots \oplus \frg\frl(n_{k-1}) \oplus \mathfrak s \mathfrak p(2n_k)$.   If instead all entries of $\lambda$ are nonnegative integers, then $\frp$ has Levi factor $\frg\frl(n_1)\oplus \cdots \oplus \frg\frl(n_k)$ if $a_k\neq 0$ and $\frg\frl(n_1)\oplus \cdots \oplus \frg\frl(n_{k-1})\oplus \frs\frp(2n_k)$ if $a_k=0$.  As before
$\frp_j$ may be defined by collapsing adjacent Levi factors for $j\in S$, and this time if $k-1\in S$ and $a_k=0$, define $\frp_{k-1}$ to have Levi factor 
$\frg\frl(n_1)\oplus \cdots \oplus \frg\frl(n_{k-2})\oplus \frs \frp(2(n_{k-1}+n_k))$. Once again one can easily verify
that  $\mathscr P = \{\frp_j \; | \; j \in S\}$ satisfies \eqref{e:key0} and \eqref{e:key}.
\qed
\end{example}

\begin{example}
\label{ex:1+N}
Continuing the example of $G = \GL(n,\bbC)$ in Example \ref{ex:classical}, if we order the set $\mathscr P$ as described there, then one may quickly verify that for $x \in \frg(-1)$,
\[
\epsilon(x) = (\mathrm{Id}_n + x) \cdot \frp.
\]
This is the map that is used in \cite{ct} to define $\caO \mapsto Q_\caO$.  But note that we may take {\em any} ordering of the set $\mathscr P$ and thus define different maps $\caO \mapsto Q_\caO$
which will have the properties that we describe in Theorem \ref{t:main}. \qed
\end{example}

\begin{example}
\label{ex:f4}
Let $G$ be simple of type F4.  Suppose $\lambda$ is one-half the middle element of a Jacobson-Morozov triple for the nilpotent orbit labeled in the Bala-Carter classification by F4(a3).  The weighted Dynkin diagram of this orbit (in the standard Bourbaki order) is 0200.   Thus $\frp{}$ is the maximal parabolic corresponding the long middle root on the Dynkin diagram of F4.  The opposite nilradical $\overline{\fru}{}$ is a three-step nilpotent algebra whose first step is $\frg(-1)$. Since $\frp{}$ is maximal, the only possibility is for $\mathscr P$ to consist of a single algebra, namely all of $\frg$.  But we have already remarked that $\frg \cap \overline{\fru}{} = \bar \fru{}$ properly contains $\frg(-1)$.  Thus there is no choice of $\mathscr{P}$ that
satisfies \eqref{e:key}.\qed
\end{example}

\begin{example}
\label{ex:g2}
Let $G$ be simple of type G2.  Suppose $\lambda$ is one-half the middle element of a Jacobson-Morozov triple for the nilpotent orbit labeled in the Bala-Carter classification by G2(a1).  Thus $\frp{}$ is the maximal parabolic corresponding to the short simple root.  Its nilradical properly contains $\frg(-1)$, so once again there is no choice of $\mathscr{P}$ that
satisfies \eqref{e:key}.  However, $\overline{\fru}{}$ in this case is a two-step nilpotent algebra whose first step is $\frg(-1)$.  We will give a related argument  in Section \ref{s:twostep} which handles this case.
\qed
\end{example}

\begin{remark}
\label{r:integral}
If $\lambda$ is not integral, one can repeat the constructions above replacing $G$ by $G(\lambda)$,
the centralizer in $G$ of $\exp(2i\pi\lambda)$.  The Lie algebra of $G(\lambda)$ is the sum of the integral
eigenspaces of $\ad(\lambda)$ and $y(\lam)$ is still an element whose square is central in $G(\lambda)$.  So the definitions above carry over without change.

Note also that above we restricted attention to the $L$ orbits on $\frg(-1)$.   One could instead consider the orbits of $L$ on an 
general eigenspace $\frg(c)$ for $c\in \bbC^\times$. 
Since
the $-1$ eigenspace of $\ad(\lambda)$ is the $c$-eigenspace of $\ad(\lambda')$ for $\lambda'=-\lambda/c$, the study of $L$ orbits on $\frg(c)$ for
$\lambda$ is equivalent to the study of $L$ orbits on $\frg(-1)$ for $\lambda'$.\qed
\end{remark}

\section{statement of main results}
\label{s:orbits}
\label{s:main}

The goal of this section is to state Theorem \ref{t:main} describing how $\caO \mapsto Q_\caO$ in Definition \ref{d:eps} preserves the singularities of the closure of $\caO$ in a precise sense.  In order to do so, we need some notation.

\subsection{Notation.}\label{s:cgp}
Suppose $X$ is a complex algebraic variety on which a complex algebraic group $H$ acts with finitely many orbits.
Let  $\caC(H,X)$ be the category of $H$-equivariant constructible sheaves on $X.$ 
Write 
 $\caP(H,X)$   for  the   category of $H$-equivariant
perverse sheaves on $X.$    

Irreducible objects in both categories are parametrized by
the set $\Xi(H,X)$  consisting of  pairs
$(Q,\caV)$ with $Q$ an orbit of $H$ on $X$ and $\caV$ an irreducible $H$-equivariant local system supported on $Q$.
For $\gamma \in \Xi(H,X)$, we write $\con(\gamma)$ and $\per(\gamma)$ for the corresponding irreducible
constructible and perverse sheaves.  

By taking Euler characteristics, we identify the
Grothendieck group of the categories $\caP(H,X)$ and  $\caC(H,X).$ In this way, we can consider the change of 
basis matrix,
\begin{equation}
\label{e:gengeomdecomp}
[\text{per}(\gamma)] =  \underset{\psi \in \Xi(H,X) }{\sum } (-1)^{d(\psi)} \;  C^g_{\psi, \gamma} [\c(\psi)];
\end{equation}
here $\psi = (Q_\psi, \caV_\psi)$ and  $d(\psi) = \dim(Q_{\psi}).$
The matrix $\left(C^g(\psi, \gamma)\right)$ is called the {geometric multiplicity matrix.} Let
$P_{\psi,\gamma} \in \bbZ[q]$ denote the graded occurrence of $\con(\psi)$ in the cohomology sheaves of $\per(\gamma)$.  More precisely, define the coefficient
of $q^i$ in $P_{\psi,\gamma}$ to be the multiplicity of $\con(\psi)$ in the $i$th cohomology sheaves of $\per(\gamma)$.
(In our applications below we will have vanishing in odd degrees.) 
Thus, up to a sign, 
\begin{equation}
\label{e:p}
P_{\psi,\gamma}(1) = C^g_{\psi, \gamma}.
\end{equation}
Finally, given an $H$ orbit $Q$ on $X$, let $m_Q^\mic$ denote the $\bbZ$-valued linear functional on the Grothendieck group of $\caP(H,X)$ that assigns to an irreducible perverse sheave the multiplicity of the conormal bundle to $Q$ in its characteristic cycle.  The notation is meant to indicate that $m_Q^\mic$ is a microlocal multiplicity.

\medskip

\subsection{Statement of Main Results}
\label{s:main}
Let $A_L(x)$ denote the component group of the centralizer in $L$ of $x \in \frg(-1)$, and let $A_K(\epsilon(x))$
denote the component group of the centralizer in $K$ of $\epsilon(x)$.  Since $\epsilon$ is $L$ equivariant, there is a natural map $A_L(x) \rightarrow A_K(\epsilon(x))$.  Since $A_K(\epsilon(x))$ is abelian (an elementary abelian 2-group) with only one-dimensional irreducible representations, composition defines a map
on irreducible representations,
\begin{equation}
\label{e:A}
A_K(\epsilon(x))^{\widehat{\phantom{x}}} \lra A_L(x)^{\widehat{\phantom{x}}}.
\end{equation}
Fix $\gamma = (Q_\caO,\caL) \in \Xi(K,Y)$, with $Q_\caO = K\cdot \epsilon(x)$ for $x \in \frg(-1)$.  Then $\caL$ is parametrized by an irreducible representation of $A_K(\epsilon(x))$.  By \eqref{e:A}, this maps to an irreducible representation of $A_L(x)$, and hence an
irreducible local system $\caL'$ on $\caO = L \cdot x$.  Define
\begin{equation}
\label{e:phi}
\varphi  \; : \; \Xi(K,Y) \lra \Xi(L,\frg(-1))
\end{equation}
by
\[
\varphi(Q_\caO,\caL) = (\caO, \caL').
\]
More conceptually, this is just the pullback via $\eps$ of irreducible local systems from $Y$ to $\frg(-1)$.
\begin{theorem}
\label{t:main}
Recall the setting of \eqref{e:QO} and the definition of $Y$ in \eqref{e:Y}.  In particular, recall that we assume that there exists a fixed ordered set of parabolic subalgebras $\mathscr P$ satisfying \eqref{e:key0}
and \eqref{e:key} that is used to define $\epsilon: \frg(-1) \ra \caP$.  Recall the definition of $\varphi$ of \eqref{e:phi}, and of the geometric multiplicity matrices and intersection homology polynomials of Section \ref{s:cgp}.  Then for 
$\psi,\gamma \in \Xi(K,Y)$,
\begin{equation}
\label{e:mainp}
P_{\psi,\gamma} = P_{\varphi(\psi), \varphi(\gamma)};
\end{equation}
in particular,
\begin{equation}
\label{e:mainc}
C^g_{\psi,\gamma} = C^g_{\varphi(\psi), \varphi(\gamma)}.
\end{equation}
Finally, fix $\caO = L \cdot x$, $Q_\caO = K\cdot \epsilon(x)$, and $\gamma \in \Xi(K,Y)$.  The multiplicity of the conormal bundle to $Q_\caO$ in the characteristic cycle of $\per(\gamma)$ equals the multiplicity of the conormal bundle to $\caO$ in $\per(\varphi(\gamma))$,
\begin{equation}
\label{e:mainCC}
m^\mic_{\caO}(\per(\varphi(\gamma))) = m^\mic_{Q_\caO}(\per(\gamma)).
\end{equation}

\end{theorem}

\noindent Before turning to the proof in Section \ref{s:proof}, we make a few remarks.

\begin{remark}
\label{r:main}
Theorem \ref{t:main} is most powerful when  the map $\varphi$ is surjective for some
choice of $\mathscr P$.  In this
case, every polynomial arising from the $L$ orbits on $\frg(-1)$ is matched with a Kazhdan-Lusztig-Vogan polynomial.  In Example \ref{ex:classical2} (extending Example \ref{ex:classical}), we sketch that this is the case for some classical subgroups of $\GL(n,\C)$.  But note that there are some cases in spin groups and exceptional groups, for example, where $A_L(x)$ is nonabelian, and therefore has higher dimensional irreducible local systems as elements of $\Xi(L,\frg(-1))$.  In these cases, $\varphi$ can never be surjective (even when a choice of $\mathscr P$ satisfying \eqref{e:key0} and \eqref{e:key} exists).
\qed
\end{remark}

\begin{remark}
\label{r:ct}
When $G = \GL(n,\bbC)$ and the choice of $\mathscr P$ (and its ordering) is the one described in Example \ref{ex:1+N}, then \eqref{e:mainc} in Theorem \ref{t:main} is \cite[Theorem 2.5]{ct}.
\qed
\end{remark}

\begin{remark}
\label{r:O}
We describe a version of Theorem \ref{t:main} that replaces $K$ orbits with $P$ orbits, and thus matches intersection homology polynomials for $L$ orbits on $\frg(-1)$ with classical parabolic Kazhdan-Lusztig polynomials.  Loosely speaking we replace every occurrence of $K$ with $P$.  In more detail, assuming the existence
of $\mathscr P$, we can define a map that takes  $\caO = L\cdot x$ for $x \in \frg(-1)$ to 
$Q_\caO = P \cdot\epsilon(\caO)$ where $\eps$ is defined as in Definition \ref{d:eps}.  If we once again set $Y$ to be the union of all the various $Q_\caO$,  then $\epsilon$ sends $\frg(-1)$ to $Y$.  Just as above, we obtain a map
\[
\varphi \; : \; \Xi(P,Y) \lra \Xi(L,\frg(-1)).
\]
The proof of Corollary \ref{c:opendense} (see Remark \ref{r:opendenseO}) will apply to show $\caO\simeq \epsilon(\caO)$ is open and dense in $ Q_\caO$.  This allows one to directly conclude
\begin{equation}
\label{e:mainpO}
P_{\psi,\gamma} = P_{\varphi(\psi), \varphi(\gamma)},
\end{equation}
and similarly match microlocal multiplicities.  (See the argument and references in Section \ref{s:opendense}.) When $G = \GL(n,\C)$ and the choice of ordering
on $\mathscr P$ is as in Example \ref{ex:1+N} (so $\epsilon(x) = \mathrm{Id}_n + x$), \eqref{e:mainpO} is exactly the main result of \cite{zel}.

Note, however, that there are no nontrivial local systems in $\Xi(P,Y)$.  So there is no hope in 
matching polynomials that arise for nontrivial local systems for $L$ orbits on $\frg(-1)$ with classical Kazhdan-Lusztig polynomials; cf. Remark \ref{r:main}.  This is a reason for using $K$ orbits in order to match more general polynomials for $\frg(-1)$. \qed  
\end{remark}

\begin{remark}
\label{r:disconnected}
Set $K' = G^\theta$ as in the introduction; so $K'$ is potentially disconnected.  In the setting of Theorem \ref{t:main}, we can copy the definitions above to define
\[
\varphi'  \; : \; \Xi(K',Y) \lra \Xi(L,\frg(-1)).
\]
Fix $\psi,\gamma \in \Xi(K',Y)$ and write $Q_\gamma$ for $K'$ orbit that is the support of $\gamma$.  If we further assume that $Q_\gamma$ is irreducible (which of course is
automatic if $K'$ is connected), then the proof of Theorem \ref{t:main} that we give below will show
\begin{equation}
\label{e:k'}
P_{\psi,\gamma} = P_{\varphi'(\psi), \varphi'(\gamma)}
\end{equation}
and that the analogous conclusion of \eqref{e:mainCC} also holds.  If we do not assume $Q_\gamma$ is irreducible, then \eqref{e:k'} can fail; see \cite[Remark 3.6]{bt} for a
discussion of an example in $\SO(8)$.
The one place in the argument below that we need $Q_\gamma$ to be irreducible is to deduce the 
conclusion of Corollary
\ref{c:opendense}.  
\end{remark}

\section{proof of theorem \ref{t:main}}
\label{s:proof}

\subsection{Preliminary Results}
\label{s:prelim}

As a first step toward Theorem \ref{t:main}, we
show $\caO \mapsto Q_\caO$ preserves dimensions; see Proposition \ref{p:gfinite} below.   
We start with some preliminaries, following the approach of \cite{bt} closely.

\begin{lemma}
\label{l:inj}
In the setting of Theorem \ref{t:main}, if $\caO = L \cdot x$ is an orbit of $L$ on $\frg(-1),$ then
\[
\dim(\epsilon(\caO)) = \dim(\caO).
\]
\end{lemma}
\begin{proof}
Since $\epsilon$ is $L$-equivariant, $Z_L(x) \subset  Z_L(\epsilon(x))$.  So the result follows from the
other containment
\begin{equation}
\label{e:supset}
Z_L(x) \supset Z_L(\epsilon(x)).
\end{equation}
Write $x = \sum_j x_i$ with possibly some of $x_j$'s equal to zero.  Because $[\frg(-1),\frg(-1)] \subset 
\frg(-2)$, there is a 
$z \in \bigoplus_{k\leq -2} \frg(k)$ so that
\begin{align*}
\epsilon(x) &:= \exp(x_1) \text{exp}(x_2) \cdots \text{exp}(x_{\ell})\cdot \frp \\
&=\exp(x_1+x_2+\cdots +x_{\ell}+ z )\cdot \frp \\
&=\exp(x + z)\cdot  \frp.
\end{align*}
If $l \in L$ centralizes $\epsilon(x)$, it thus centralizes $x+z$.  Since $L$ preserves the grading of $\frg = \bigoplus_k \frg_k$, if $l\in L$ centralizes $x +z$, it must centralize
$x$, and so \eqref{e:supset} follows.
\end{proof}
\medskip

\begin{lemma} \label{l:freeaction}
In the setting of Theorem \ref{t:main}, write $\bar P = L\bar U$ for the opposite parabolic subgroup to $P=LU$.  Then $\bar{U}\cap K$ acts  freely on 
\[
[\bar{U} \cap K] \cdot \epsilon (\frg(-1)).
\]
Moreover, for all $ x \in \frg(-1)$,
\[
\left( [\bar{U} \cap K] \cdot \epsilon (x) \right ) \cap \epsilon(\frg(-1)) = \epsilon(x).
\]
In particular, $\eps$ is injective.
\end{lemma}
\begin{proof} Suppose $k \in \bar{U}\cap K$ and
$x = \sum x_j \in \frg(-1)$ such that 
\begin{equation}\label{inj}
k \cdot \text{exp}(x_1)\; \text{exp}(x_2)\cdots \text{exp}(x_{\ell}) \cdot \frp = \text{exp}(x_1) \; \text{exp}(x_2) \cdots\text{exp}(x_{\ell}) \cdot \frp.
\end{equation}
The stabilizer in $G$ of $\frp$ is $P$ and $\bar{U} \cap P= {1}.$   Thus \eqref{inj} implies
$$k \cdot \text{exp}(x_1)\; \text{exp}(x_2)\; \ldots  \text{exp}(x_{\ell}) = \text{exp}(x_1) \; \text{exp}(x_2)
 \ldots  \text{exp}(x_{\ell}),$$
from which we conclude that $ k = 1$, verifying the first assertion of the lemma. The second assertion follows in a similar way.
\end{proof}

\bigskip
\begin{prop}\label{p:gfinite}
In the setting of Theorem \ref{t:main}, let $\caO$ be an orbit of $L$ on $\frg(-1)$ and define $Q_\caO$ as in \eqref{e:QO}.  Then, 
\begin{equation}
\label{e:gfinite}
\mathrm{dim}(Q_\caO) = \mathrm{dim}(Q_{\{0\}}) +  \mathrm{dim}(\caO).
\end{equation}
\end{prop}

\begin{proof} 
We first show that
\begin{equation}
\label{e:dim1}
\dim(Q_{\caO})\leq \dim(Q_{\{0\}}) + \text{dim}(\caO) .
\end{equation}
To see this, we factor $\epsilon$ as follows.  
Write $P_i$ for the centralizer in $G$ of $\frp_i$.
Define
\begin{equation}
\label{e:X}
\caX  = K\underset{K\cap P}{\times}P_{1}\underset{ P}{\times}P_{2}\underset{ P}{\times}\ldots \underset{ P}{\times}P_{\ell}
\end{equation}
to be the quotient of  $K\times P_{1}\times \ldots \times P_{\ell}$ by the action
\[ (p_0, p_{1}, \ldots, p_{\ell}) \cdot  (k_0, y_{1},  y_{{2}}, \ldots , y_{\ell}) =
(k_0 \;p_0, p_0^{-1}\; y_{1} \;p_{1},   p^{-1}_{1}  \;  y_{{2}} \;p_{2} , \ldots, p_{\ell}^{-1}\; y_{\ell}).
\]
Let $K$ act on $\caX$ via 
\[ k \cdot [k_0, y_{1},  y_{{2}}, \ldots , y_{\ell}] =   [ k \;k_0, y_{1},  y_{{2}}, \ldots , y_{\ell}].\]
Then $\caX$ comes equipped with a natural $K$-equivariant map
\begin{equation}
\label{e:tau}
\tau \; : \; \caX \lra \caP
\end{equation}
mapping
\[
[k_0, y_{1},  y_{{2}}, \ldots , y_{\ell}] \mapsto k_0y_1\cdots y_\ell \cdot \frp.
\]
Define an $L$ equivariant map
\[
\iota \; : \; \frg_{-1}  \longrightarrow \caX 
\]
mapping $x = x_1 + \cdots + x_\ell$ as 
\[
\iota(x) = [1, \text{exp}(x_1), \text{exp}(x_2), \ldots, \text{exp}(x_{\ell})].
\]
Then, by definition, $\epsilon = \tau \circ \iota$ and
\[
Q_\caO = \tau(K\cdot \iota(\caO)).
\]
Thus
\[
\dim(Q_\caO) \leq \dim(K\cdot \iota(\caO))\leq \dim(K/K\cap P) + \dim(\caO).
\]
Once we note that $Q_{\{0\}} = K\cdot \frp \simeq K/K\cap P$, \eqref{e:dim1}
follows.

We argue that the converse inequality holds. 
Since $\bar{U}\cap K \cdot \epsilon(\caO_{})$ is contained in $Q_{\caO},$
 Lemma \ref{l:freeaction} implies  
\begin{equation}
\label{e:step1}
\dim(Q_{\caO}) \geq \dim(\bar{U}\cap K) + \dim(\epsilon(\caO)).
\end{equation}
By Lemma \ref{l:inj}, we know that  $\dim(\epsilon(\caO)) = \dim(\caO).$  Since
 $Q_{\{0\}} = K \cdot \frp  \simeq K/(K \cap P)$,  
\[
\dim(Q_{\{0\}})  = \dim(\bar{U}\cap K).
\]
Thus, \eqref{e:step1} becomes
\[
\dim(Q_\caO) \geq \dim(Q_{\{0\}}) +  \dim(\caO),
\]
as we wished to show.
\end{proof}
\medskip

\begin{corollary}\label{c:opendense}
$[K\cap \bar{P}] \cdot \epsilon (x)$ is open and dense in $K \cdot  \epsilon (x)$. 
\end{corollary}
\begin{proof}
Since $\dim(Q_{\{0\}}) = \dim(K/(K\cap P))=\dim(\bar U \cap K)$,
Lemma \ref{l:freeaction} and the dimension count of Proposition \ref{p:gfinite} implies the result.
\end{proof}

\medskip

\begin{remark}
\label{r:opendenseO}
We can repeat the analysis above with $K$ replace by $P$ (as in Remark \ref{r:O}).  In this
case $Q_{\{0\}}$ is just the point $\frp$ in $\caP$.  Replacing $K$ with $P$ in the proofs of
Proposition \ref{p:gfinite} and Corollary \ref{c:opendense} implies that
\[
\dim(Q_\caO) = \dim(\caO),
\]
and $\epsilon({\caO})$ is open and dense in ${Q_\caO}$.\qed
\end{remark}

\subsection{Consequences of Corollary \ref{c:opendense}.}
\label{s:opendense} 
Recall that the component group of the centralizer in $K$ of an element of $\caP$ is an elementary abelian 2-group.  Thus, every irreducible $K$-equivariant local system on $\caP$ is one-dimensional.  In the setting of Corollary \ref{c:opendense} and notation for $Y$ in \eqref{e:Y}, we thus obtain a restriction map,
\begin{equation}
\label{e:phi1}
\varphi_1 \; : \; \Xi(K,Y) \lra \Xi \left (K\cap \bar{P}, [K\cap \bar P] \cdot \epsilon(\frg(-1)) \right ).
\end{equation}
This is once again dual to the corresponding restriction of $A$-group representations as
in \eqref{e:A}.  By definition, if $\gamma \in \Xi(K,Y)$, then the constructible sheaf $\con(\gamma)$ restricts to $\con(\varphi_1(\gamma))$.
By the density statement, the intersection homology polynomials for $K \cap \bar P$ orbit closures on $[K\cap \bar P] \cdot \epsilon(\frg(-1))$ match those for $K$ orbits on $Y$. (This is part of the unicity of perverse extensions.  A discussion of such a statement can be found in the proof of Proposition 7.14(c) and around Equation (7.16)(e) in \cite{abv}.) More
precisely, for $\psi, \gamma \in \Xi(K,Y)$,
\[
P_{\psi,\gamma} = P_{\varphi_1(\psi), \varphi_1(\gamma)};
\]
in particular,
\[
C^g_{\psi,\gamma} = C^g_{\varphi_1(\psi), \varphi_1(\gamma)}.
\]
Finally, using the method of calculating of characteristic cycles via normal slices (sketched, for example, at the bottom of page 186 and top of page 187 in \cite{abv}), one sees that the microlocal multiplicities match: if $\caO = K\cdot \epsilon(x)$, $\caO'=[K\cap \bar P]\cdot \epsilon(x)$ and $\gamma \in \Xi(K,Y)$,
\[
m_{\caO'}(\per(\gamma)) = m_{\caO}(\per(\varphi_1(\gamma)).
\]

\subsection{Induced bundles}
\label{s:ib}
The previous section relates the geometry of orbits of $K$ on $Y$ to the orbits of $K\cap \bar P$ on $[K\cap \bar P] \cdot \epsilon(\frg(-1))$.  In this section, we use an induced bundle construction to relate these $K \cap \bar P$ orbits to $L$ orbits on $\frg(-1)$.

To begin, recall a general construction.  Suppose $H$ acts on a variety $X$ with finitely many orbits.  Suppose $H \subset H'$.  The induced bundle
\[
H' \times_H X
\]
is defined by quotienting $H' \times X$ by $(h'h,x) \sim(h',hx)$ for all $h \in H$.  See \cite[Chapter 7]{abv}, for example.

\label{sec:ib}
\begin{prop}\label{p:ib}
In the setting of Section \ref{s:orbits}, the map
\[
\IB : [K  \cap \bar P] \times_{L} \frg(-1) \lra [K \cap \bar P] \cdot \epsilon (\frg(-1))
\]
defined by
\[
\IB(k, x ) \mapsto k \;\epsilon (x)
\]
is a $K \cap \bar P$ equivariant isomorphism.
\end{prop}
\begin{proof}
Clearly $\IB$ is surjective and $K \cap \bar P$ equivariant.  We prove that $\IB$ is injective. Suppose $\IB(k_1, x_1) = \IB(k_2,x_2) .$ 
Write $k_1 = \bar{n}_1\;  t_1$ with $\bar{n}_1\in [\bar{U}\cap K] $ and $t_1 \in L \cap K.$
Similarly, write  $k_2 = \bar{n}_2\;  t_2.$
Thus
\[
\bar n_1 t_1 \epsilon(x_1) = \bar n_2 t_2 \epsilon(x_2).
\]
Since $\eps$ is $L$ equivariant, this implies
\begin{equation}\label{a}
\bar n_1 \epsilon(\Ad(t_1)x_1) = \bar n_2 \epsilon(\Ad(t_2)x_2).
\end{equation}
By Lemma \ref{l:freeaction},  $\bar n_1 = \bar n_2$ and
\begin{equation}
\label{e:eps-eq}
 \epsilon(\Ad(t_1)x_1) = \epsilon(\Ad(t_2)x_2).
\end{equation}
Since $\eps$ is injective (Lemma \ref{l:freeaction}), $\Ad(t_1) x_1 = \Ad(t_2) x_2$.  We use this in the third equality below to conclude,
\begin{align*}
(k_2, x_2) &=
 (\bar n_2 t_2,x_2) =  (\bar n_2t_2,\Ad(t_2)^{-1}\Ad(t_2)x_2) = (\bar n_2 t_2, \Ad(t_2)^{-1}\Ad(t_1)x_1) \\
 &\sim(\bar n_2, \Ad(t_1)x_1) = (\bar n_1, \Ad(t_1)x_1)\\
&=(k_1t_1^{-1},\Ad(t_1)x_1) \sim (k_1,x_1).
\end{align*}
Thus $\IB(k_1,x_1) = \IB(k_2,x_2)$ implies $(k_1, x_1) \sim (k_2,x_2)$, and so $\IB$ is injective as we wished to show.
\end{proof}

\medskip

\begin{cor}
\label{c:ib-matching}
In the setting of Proposition \ref{p:ib}, there is a 
is a natural correspondence of
$L$ orbits on $\frg(-1)$ and $K\cap \bar P$ orbits on $[K\cap \bar P]\cdot\epsilon(\frg(-1))$,
\[
\caO = L\cdot x \mapsto \caO' = [K \cap \bar P] \cdot \epsilon(x).
\]
Write $A_L(x)$ for the component group of the 
centralizer of $x$ in $L$, and similarly for $A_{K \cap \bar P}(\epsilon(x))$.  Then the map
$L \rightarrow [K \cap \bar P]$ induces
an isomorphism
\[
A_L(x) \simeq A_{K\cap \bar P}(\epsilon(x)).
\]
The resulting bijection
\begin{equation}
\label{e:phi2}
\varphi_2 \; : \; \Xi(K\cap \bar P, [K\cap \bar P]\cdot\epsilon(\frg(-1))) \rightarrow \Xi(L,\frg(-1))  
\end{equation}
implements an identification of the geometric mulitplicity matrices and intersection homology polynomials of Section \ref{s:cgp}.  
More precisely, for $\psi,\gamma$ in $\Xi(K\cap \bar P, [K\cap \bar P]\cdot\epsilon(\frg(-1)))$
\begin{equation}
\label{e:ib-matchingP}
P_{\psi,\gamma} = P_{\varphi_2(\psi),\varphi_2(\gamma)} \; 
\end{equation}
and, in particular,
\begin{equation}
\label{e:ib-matchingC}
C^g_{\psi,\gamma} = C^g_{\varphi_2(\psi),\varphi_2(\gamma)}.
\end{equation}
Finally, the microlocal multiplicities match,
\begin{equation}
\label{e:ib-matching-mic}
m_{\caO'}(\per(\gamma)) = m_{\caO}(\per(\varphi_2(\gamma)).
\end{equation}
\end{cor}

\begin{proof}
According to \cite[Proposition 7.14]{abv}, 
there is a bijective correspondence of $L$ orbits on $\frg(-1)$ and $K\cap\bar P$ orbits
on the induced bundle $ [K  \cap \bar P] \times_{L} \frg(-1) $ with properties as listed in \eqref{e:phi2}-\eqref{e:ib-matchingC}, while \cite[Proposition 20.1(e)]{abv} implies \eqref{e:ib-matching-mic}. Composing with the isomorphism of Proposition \ref{p:ib} completes the proof.
\end{proof}

\medskip

\subsection{Proof of Theorem \ref{t:main}}
\label{s:mainproof}
Once one observes that $\varphi$ in Theorem \ref{t:main} is simply the composition of $\varphi_2$ in \eqref{e:phi2} and
$\varphi_1$ in \eqref{e:phi1},
 the theorem follows by combining the results of Section \ref{s:opendense} and Corollary \ref{c:ib-matching}.\qed

\medskip
\begin{remark}
\label{ex:classical2}
Determining when $\varphi$ is surjective requires case-by-case analysis, and we simply record the results here for classical groups.    If $G=\GL(n)$, then all local systems are trivial, so $\varphi$
is automatically surjective for any ordering of the set $\mathscr P$ defined in Example \ref{ex:classical}. As \cite[Example 3.5]{bt} already indicates, for some choices of ordering of $\mathscr P$, $\varphi$ can fail to be surjective.  But there is always some choice for which $\varphi$ {\em is} surjective.  If $G = \Sp(2n)$, and $\lambda$ consists of all integers, there are again no nontrivial local systems, so surjectivity is automatic.  If, however, $\lambda$ consists of half-integers, $\mathscr P$ in Example \ref{ex:classical} sometimes contains a distinguished element $\frp_\circ$.  This element must appear first in the ordering on $\mathscr P$ in order for $\varphi$ to be surjective.  The situation is similar for $\SO(n)$ where there is sometimes a distinguished element in the set $\mathscr P$ defined in Example \ref{ex:classical} which has a Levi factor component of the form $\frs \fro(n)$.  Again, this element must be taken first in the order on $\mathscr P$ in order for $\varphi$ to be surjective. \qed \end{remark}

\section{abelian and two-step case}
\label{s:twostep}

In this section, we study a special class of examples that includes some cases (like Example \ref{ex:g2}) where there does not exist a set 
$\mathscr P$ satisfying \eqref{e:key0} and \eqref{e:key}. In these cases, we do not need to truncate
the exponential map.  Alternatively, one can think of this section as a kind of basic case, and the definition of 
$\eps$ in Definition \ref{d:eps} and some of the proofs of Section \ref{s:proof} as an induction using the bundle constructed in \eqref{e:X}.  This can be made
precise, but isn't necessary for our purposes here.

Assume that the $i$-eigenspace $\frg(i)$ is zero if $|i| >2$.  In other words, the nilradical
of $\frp$ is either abelian or a two-step nilpotent Lie algebra. Under this hypothesis,
we simply define
\begin{equation}
\label{emb}
\epsilon' :  \frg(-1) \longrightarrow   \caP
\end{equation}
by
\[
\epsilon'(x) = \exp(x)\cdot \frp.
\]
Using $\epsilon'$ instead of $\epsilon$ in \eqref{e:QO}, for an $L$ orbits $\caO$
on $\frg(-1)$ we define
\[
Q'_\caO = K \cdot \epsilon'(\caO),
\]
and $Y'$ to be the union of the various $Q'_\caO$.
As in \eqref{e:A}, for $x \in \frg(-1)$ we have a natural map
\begin{equation}
\label{e:A2}
A_K(\epsilon'(x))^{\widehat{\phantom{x}}} \lra A_L(x)^{\widehat{\phantom{x}}},
\end{equation}
and hence a map
\begin{equation}
\label{e:phi2step}
\varphi'  \; : \; \Xi(K,Y') \lra \Xi(L,\frg(-1)).
\end{equation}
We then have the following analog of Theorem \ref{t:main}.

\begin{theorem}
\label{t:main2}
For 
$\psi,\gamma \in \Xi(K,Y')$,
\begin{equation}
\label{e:mainp2}
P_{\psi,\gamma} = P_{\varphi'(\psi), \varphi'(\gamma)}.
\end{equation}
Fix $\caO = L \cdot x$, $Q'_\caO = K\cdot \epsilon'(x)$, and $\gamma \in \Xi(K,Y')$.  Then
\[
m^\mic_{\caO}(\per(\varphi'(\gamma))) = m^\mic_{Q'_\caO}(\per(\gamma)).
\]
\end{theorem}

\smallskip

\noindent This will follow in exactly the same way as Theorem \ref{t:main} once we prove analogs of Lemma \ref{l:freeaction}, Proposition \ref{p:gfinite} and Corollary \ref{c:opendense}.

\begin{lemma} \label{l:freeaction2}
 The group  $\bar{U}\cap K$ acts  freely on 
\[
[\bar{U} \cap K] \cdot \epsilon'(\frg(-1)).
\]
Moreover, for all $x \in \frg(-1)$,
\[
\left( [\bar{U} \cap K] \cdot \epsilon' (x) \right ) \cap \epsilon'(\frg_{-1}) = \epsilon'(x).
\]
\end{lemma}
\begin{proof} This follows as in the proof of Lemma \ref{l:freeaction}. \end{proof}
\medskip
\begin{prop}\label{p:gfinite2}
Let $\caO$ be the $L$ orbit of $x \in \frg(-1)$.  Then, 
\begin{equation}
\label{e:gfinite2}
\mathrm{dim}(Q'_\caO) = \mathrm{dim}(Q'_{\{0\}}) +  \mathrm{dim}(\caO).
\end{equation}
\end{prop}

\begin{proof}
Let $\frg = \frk \oplus \frs$ denote the Cartan decomposition.  Because of our two-step 
assumption,  $\frs = \frg(-1) \oplus \frg(1)$, and $\frk = \frg(-2) \oplus \frl \oplus \frg(2)$.
The tangent space to $Q'_{\caO}$ at $\frq := \epsilon'(x) = \exp(x)\cdot\frp$ identifies with 
$\frg/(\frk+\frq)$.  If we fix a nondegenerate invariant form to identify $\frg$ and $\frg^*$,
then then conormal space to $Q'_{\caO}$ at $\frq$ identifies with
\[
[\frg/(\frk+\frq)]^{*,\perp} \simeq \Ad(\exp(x))\bar \fru \cap \frs.
\]
Since $\bar \fru = \frg(-2) \oplus \frg(-1)$, $x \in \frg(-1)$, and $[\frg_i,\frg_j] \subset \frg_{i+j}$, 
\[
\Ad(\exp(x))\bar \fru \cap \frs = \mathrm{ker}(\ad(x)\vert_{\frg(-1)}).
\]
The dimension of $\caP$ is simply $\dim(\bar \fru)$.  On the other hand, $\dim(\caP)$ equals the
the sum of the dimension of $Q'_\caO$ and the dimension of the conormal space to $Q'_\caO$ at $\frq$.  Thus
\[
\dim (\bar{\fru})  = \text{dim}(Q'_\caO)  + \mathrm{ker}(\ad(x)\vert_{\frg(-1)}).
\]
So
\begin{align*}
\text{dim}(Q'_\caO) & = \dim (\bar{\fru}) - \mathrm{ker}(\ad(x)\vert_{\frg(-1)}) \\
& = \dim(\frg(-2)) + \left[ \text{dim} (\frg(-1)) - \mathrm{ker}(\ad(x)\vert_{\frg(-1)})\right]\\
& = \dim(\frg(-2))  + \text{dim}(\caO).
\end{align*}
To complete the proof, note that $Q_{\{0\}} = K\cdot \frp$ has dimension equal to that
of $\frk/(\frk \cap \frp) \simeq \frg(-2)$.
\end{proof}

\medskip

\begin{cor}\label{c:opendense2} $[K\cap \bar{B}] \cdot \epsilon'(x)$ is open and dense in $K \cdot  \epsilon'(x)$.  
\end{cor}

\begin{proof}
This now follows in the same way as Corollary \ref{c:opendense}.
\end{proof}

\medskip

\noindent With Corollary \ref{c:opendense2} in hand, we can argue as in the proof of Theorem \ref{t:main}
to establish Theorem \ref{t:main2}.  We omit the details.

\section{remarks on  abv micro-packets}
\label{s:abv}
To conclude, we sketch some consequences related to micro-packets of representations.  Part of our motivation is to explain the appearance of the Kashiwara-Saito 
singularity discovered by Cunningham-Fiora-Kitt \cite{cunningham}.

Fix $\lambda$ integral and semisimple as above.  According to the \cite{vogan:ic4} (reinterpreted in \cite[Theorem 1.24]{abv}), $\Xi_0(K,Y)$ parametrizes a subset of irreducible 
representations with infinitesimal character $\lambda$ of various real forms $G_\bbR^{\prime}$ of the Langlands dual group $G^\vee$.  (In order to think of $\lambda$ as an infinitesimal character for $G_\bbR^{\prime}$, we  view $\lambda \in \frh$ as an element of $(\frh^\vee)^*$ for the dual Cartan subalgebra $\frh^\vee \simeq \frh^*$.) Write $\pi_\bbR(\gamma)$ for the irreducible
representation corresponding to $\gamma \in \Xi(K,Y)$, and let $\Pi_Y$ denote all of the representations of the form $\pi(\gamma)$.

Fix a $K$ orbit $Q$ on $Y$.  Following \cite[Definition 19.15]{abv} define
a subset $\Pi_\R^g(Q)$ of $\Pi_Y$ consisting of those $\pi_\bbR(\gamma)$ such that
the conormal bundle to $Q$ occurs in  the characteristic cycle of $\per(\gamma)$.  The subset  $\Pi_\R^g(Q)$ is called a micro-packet.   Arthur packets are defined in \cite[Definition 22.6]{abv}  as certain special kinds of micro-packets arising from Arthur parameters.

Meanwhile, if we assume further that $\lambda$ is hyperbolic, then according to Lusztig's classification of unipotent representation of graded affine Hecke algebras (\cite{lu:cls1, lu:graded, lu:cls2}), the set $\Xi(L,\frg(-1))$ parameterizes certain unipotent representations of the split $F=\bbQ_p$ form $G'_F$ of
the Langlands dual group $G^\vee$.  (The subset $\Xi_0(L,\frg(-1))$ consisting of irreducible local systems of Springer type parametrizes irreducible unramified representations of $G'_F$.)  Write $\pi_F(\gamma)$ for the irreducible unipotent representation
of $G'_F$ corresponding to $\gamma \in \Xi(L,\frg(-1))$.  Fix an $L$ orbit $\caO$ on $\frg(-1)$.  Just as above, we can
define a micro-packet $\Pi^g_F(\caO)$ consisting of those $\pi_F(\gamma)$ such that
the conormal bundle to $\caO$ occurs in  the characteristic cycle of $\per(\gamma)$.  Vogan has proposed a definition of Arthur packets as a certain special kind of micro-packets in the $p$-adic case.  See \cite{v:llc} and \cite{cfmmx}.

Because $\varphi$ preserves microlocal multiplicities according to Theorem \ref{t:main}, we  conclude that $\varphi$ 
of \eqref{e:phi} maps micro-packets for real groups into micro-packets for $p$-adic groups.  More precisely, fix $\Pi_\R^g(Q_\caO) \subset \Pi_Y$.  Then
\begin{equation}
\label{e:packet}
\left \{ \pi_F(\varphi(\gamma)) \; | \; \pi_\bbR(\gamma) \in \Pi_\R^g(Q_\caO) \right \}
\end{equation}
is contained in micro-packet for $G'_F$ parametrized by $\caO$; if $\varphi$ is surjective, then this will be the entire packet.

For example, consider $G = \GL(n,\bbC)$.  In \cite{bst}, we found occurrences of the
Kashiwara-Saito singularity in $K$ orbits on $\caP$, and deduced the existence of 
reducible characteristic cycles.  In the notation above, this gives examples of micro-packets of the form $\Pi^g_\bbR(Q)$ with more than one element for $\GL(n,\bbR)$.  From
the proof of Theorem \ref{t:main}, one immediately deduces the occurrence of the Kashiwara-Saito singularity in the closure of an $L$ orbit on $\frg(-1)$, and the existence of micro-packets for $\GL(n,F)$ with more than one element, as discovered by \cite{cunningham}.  (It seems plausible that all of the interesting examples in \cite{cunningham} are 
 accounted for by matching with the real case as the choice of ordering on $\mathscr P$ varies.)  In any event, since the theory in the real case is better
developed, it is interesting to study other classical $p$-adic cases from this viewpoint.  We would like to return to this elsewhere.

Finally, one especially interesting class of micro-packets for real groups are the special unipotent
Arthur packets of \cite[Chapter 27]{abv}.  The construction of \eqref{e:packet} applied to special unipotent packets for  real classical groups should give rise to what one might call special unipotent packets for split $p$-adic classical groups.  
It would be interesting to compare this notion with the definition recently
given by Ciubotaru, Mason-Brown, and Okada in \cite{cmbo}.

\end{document}